\RequirePackage{fix-cm}

\documentclass[11pt]{svjour3}
\usepackage{amssymb}
\usepackage{amsmath, mathtools}
\usepackage[11pt]{extsizes}
\smartqed  % flush right qed marks, e.g. at end of proof
\usepackage{booktabs,caption}

\usepackage{tikz}
\usepackage{calc}
\usetikzlibrary{decorations.markings}
\tikzstyle{vertex}=[circle, draw, inner sep=0pt, minimum size=1pt]
\newcommand{\vertex}{\node[vertex]}

\usepackage{graphicx}
\usepackage{pdflscape}
\usepackage{geometry}
\numberwithin{equation}{section}

\usepackage{amsfonts}
\usepackage{xcolor}

\journalname{Indian J. Pure Appl. Math.}
\begin{document}

\title{Roman domination number of zero-divisor graphs over commutative rings}

\titlerunning{Roman domination number of zero-divisor graphs over commutative rings}        % if too long for running head

\author{Ravindra Kumar$^1$\and Om Prakash$^{*1}$  %etc.
}
\authorrunning{R. Kumar and O. Prakash}

\institute{\at $^1$
              Department of Mathematics\\
              Indian Institute of Technology Patna, Patna 801 106, India \\
              \email{om@iitp.ac.in(*corresponding author), ravindra.pma15@iitp.ac.in  }
          }
\date{Received: date / Accepted: date}
% The correct dates will be entered by the editor
\maketitle
\begin{abstract}
 For a graph $G= (V, E)$, a Roman dominating function is a map $f : V \rightarrow \{0, 1, 2\}$ satisfies the property that if $f(v) = 0$, then $v$ must have adjacent to at least one vertex $u$ such that $f(u)= 2$. The weight of a Roman dominating function $f$ is the value $f(V)= \Sigma_{u \in V} f(u)$, and the minimum weight of a Roman dominating function on $G$ is called the Roman domination number of $G$, denoted by $\gamma_R(G)$. The main focus of this paper is to study the Roman domination number of zero-divisor graph $\Gamma(R)$ and find the bounds of the Roman domination number of $T(\Gamma(R))$.

\end{abstract}

\section{Introduction}
Let $R$ be a commutative ring with unity and $Z(R)$ be the set of zero-divisors of $R$. The zero-divisor graph of $R$, denoted by $\Gamma(R)$, is a graph with set of vertices $Z(R)- \{0\}$ such that there is an edge (undirected) between the vertices $x, y \in V(\Gamma(R))$ if and only if $xy = 0$. It is noted that $\Gamma(R)$ is an empty graph if and only if $R$ is an integral domain.\\
\indent The concept of the zero-divisor graph was introduced by Beck in \cite{beck} in 1988. Later, Anderson and Livingston \cite{ander} redefined Beck's definition in 1999 and established several fundamental results on $\Gamma(R)$. Consequently, in the last four decades, plenty of works have been reported by several researchers, a few are \cite{akbari1,akbari,ander,beck,kumar,kumar1}. Further, in $2002$, Redmond \cite{redmond} extended the study of zero-divisor graph for noncommutative rings. He defined an undirected zero-divisor graph $\Gamma(R)$ of a noncommutative ring $R$ with set of vertices $Z(R)^* = Z(R) \setminus \{0\}$ and for distinct vertices $a$ and $b$, there is an edge between them if and only if either $ab= 0$ or $ba= 0$.

\par On the other hand, the concept of the Roman domination was motivated by the defence strategies used to defend the Roman empire during the reign of Emperor Constantine the great $274-337$ AD. There were mainly eight region from Asia minor to Britain of Roman empire at the time of Constantine. To defend all the region by the four groups of legions, he imposed the certain rules. He ordered that for all cities of the Roman empire, at most two group of legions should be stationed under following conditions.

\begin{itemize}
	\item {A region is securable if a group of legion can be moved to it in a single step from an adjacent region.}
	
	\item {At least two group of legions must occupy a region before a group of legion can move out of it (i.e., at least one group of legion must remain behind).}
\end{itemize}

Based on the above conditions of the Roman Empire, presently we have the mathematical concept of Roman domination. It is initially defined and discussed by Stewart \cite{stewart} in $1999$, and later by ReVelle and Rosing \cite{revelle} in $2000$. The proper definition of Roman domination was introduced by Cockayne et al. \cite{cockayne} in 2004. After that several works have been reported on various aspects of Roman domination in the graph, including generalizations such as weak Roman domination \cite{henning}, double Roman domination \cite{ahangar, beeler}.

\par A Roman dominating function on a graph $G= (V, E)$ is a function $f : V \rightarrow \{0, 1, 2\}$ with the property that every vertex $u \in V$ for which $f(u) = 0$ is adjacent to at least one vertex $v \in V$ for which $f(v)= 2$. The weight of a Roman dominating function is the value $f(V)= \Sigma_{u \in V} f(u)$. The Roman domination number of a graph $G$, denoted by $\gamma_R(G)$, is the minimum weight of an Roman dominating function on a graph $G$. Further, let $G= (V, E)$ be a graph with $f : V \rightarrow \{0, 1, 2\},$ a function and $V_0, V_1, V_2$ be the ordered partition of $V$ induced by $f$, where $V_i= \{v \in V \vert ~ f(v)= i \}$ and $\lvert V_i \rvert = n_i$, for $i = 0, 1, 2$. It is noted that there exists a one-one correspondence between the function $f : V \rightarrow \{0, 1, 2\}$ and the ordered partition $V_0, V_1, V_2$ of $V$. Therefore, it can be represented as $f = (V_0, V_1, V_2)$. A function $f = (V_0, V_1, V_2)$ is a Roman dominating function (RDF) if the set $V_2$ dominates the set $V_0,$ i.e., $V_0 \subseteq N[V_2]$. A function $f = (V_0, V_1, V_2)$ is said to be a $\gamma_R$-function if it is an RDF and $f(V)= \gamma_R(G)$.

\par Now, we recall some definitions and notations that will be used throughout this paper. Let $G = (V, E)$ be a graph of order $n$. The open neighbourhood of any vertex $v \in V$ is the set $N(v) = \{u \in V \vert uv \in E\}$ and closed neighbourhood is the
set $N[V] = N(v) \bigcup \{v\}$. The open neighbourhood of a subset $S$ of $V$ is $N(S) = \bigcup_{v \in S}N(v)$ and the closed neighbourhood is $N[S] = N(S) \bigcup S$. A set $S \subseteq V$ is called a dominating set if every vertex of $V$ is either in $S$ or adjacent to at least one vertex in $S$. The domination number $\gamma(G)$ of a graph $G$ is the minimum cardinality among the dominating sets of $G$. A graph $G$ of order $n$ is said to be complete if every vertex in $G$ is adjacent to every other vertex in $G$ and it is denoted by $K_n$. A graph is said to be regular or $k$-regular if all its vertices have the same degree $k$. Also, a graph $G=(V, E)$ is called a bipartite graph if its vertex set $V$ can be partitioned into two subsets $V_1$ and $V_2$ such that each edge of $G$ has one end vertex in $V_1$ and another end vertex in $V_2$. It is denoted by $K_{m,n}$ where $m$ and $n$ are the numbers of vertices in $V_1$ and $V_2$, respectively. A complete bipartite graph of the form $K_{1,n}$ is called a star graph. For more basic definitions and results on graph theory, we may refer \cite{bala}.

\par Section $2$ contains some basic results on Roman domination graph. In section $3$, we present Roman domination number of a zero-divisor graph $\Gamma(R)$ for $R = R_1 \times R_2$ for different diameters of $R_1$ and $R_2$ and later we generalized it for $R= R_1 \times R_2 \times ... \times R_n$. In section $4$, we present lower and upper bounds for the Roman domination number of $T(\Gamma(R))$. Section 5 concludes the work.

\section{Basic Results}
We start this section with several classes of graphs with well-known Roman domination numbers and their straightforward calculations.
\par It is easy to see that for a complete graph $K_n$, $\gamma_R(K_n) = 2$. Let $G$ be a complete $r-$ partite graph $(r\geq 2)$ with partite set $V_1, V_2,...,V_r$ such that $\lvert V_i \rvert > 2$ for $1\leq i\leq r$. Then $\gamma_R(G) = 4$. If $\lvert V_i \rvert = 2$ for some $i$, then $\gamma_R(G) = 3$ because one vertex of that set assigned $2$ and another vertex is assigned $1$. If $\lvert V_i \rvert = 1$ for some $i$, then $\gamma_R(G) = 2$. Hence, we can say that Roman domination number of any star graph is $2$ and bistar graph is $3$.\\

\begin{example}
	Consider a ring $R = \mathbb{Z}_{25}$. The graph of $\Gamma(\mathbb{Z}_{25})$ is shown in figure $1$.
	
	\[\begin{tikzpicture}

	\vertex (A) at (-1.5,1.5) [label=left:${5}$]{};
	\vertex (B) at (0.5,1.5) [label=right:${10}$]{};
	\vertex (C) at (0.5,-0.5) [label=right:${15}$]{};
	\vertex (D) at (-1.5,-0.5) [label=left:${20}$]{};
	\path
	(A) edge (B)
	(B) edge (C)
	(C) edge (D)
	(D) edge (A)
	(A) edge (C)
	(B) edge (D)
	;
	
	\end{tikzpicture}\]
	\hspace{6.6cm} \textbf{Figure 1}  \\	
	
\end{example}

In this case, the graph $\Gamma(\mathbb{Z}_{25})$ is a complete graph of $4$ vertices i.e.$K_4$. Now, we define a function $g: V(\Gamma(\mathbb{Z}_{25})) \longrightarrow \{0, 1, 2\}$ in a way such that $g(5) = 0$, $g(10) = 0$, $g(15) = 0$ and $g(20) = 2$. Clearly, by the definition, $g$ is an RDF with weight $g(V) = \sum_{u\in V} f(u) =2$. Since, this weight is minimum, so $\gamma_R(\Gamma(\mathbb{Z}_{25})) = 2$ or, $\gamma_R(K_n) = 2$.

\par Moreover, some results on Roman domination number given by Cockayne et al. in \cite{cockayne} are given below.

\begin{proposition}
	For the classes of paths $P_n$ and cycles $C_n$,
	\par 	  $\gamma_R(P_n) = \gamma_R(C_n) = \lceil \frac{2n}{3} \rceil$.\\
	
\end{proposition}

Also, they have proposed a relation between domination number and Roman domination number of a graph as follows.

\begin{proposition}
	For any graph $G$,\\
	\par	$\gamma(G) \leq \gamma_R(G) \leq 2\gamma(G)$.
	
\end{proposition}

\begin{proposition}
	For any graph $G$ of order $n$, $\gamma(G) = \gamma_R(G)$ if and only if $G = \overline{K_n}$.
	
\end{proposition}

\section{Main Results}

\begin{theorem}
	Let $S$ be a finite principal ideal local ring. Then $\gamma_R(\Gamma(S)) = 2$.
	
\end{theorem}

\begin{proof}
	Let $M$ be a maximal ideal of the finite principal ideal local ring $S$. Suppose $a \in S$ such that $M = <a>$, then $M= aS$. Let the set of unit elements of $S$ be $U= \{ u_1, u_2,..., u_m \}$. Since $S$ finite, there exists a positive integer $n$ such that $a^n = 0$ and $a^{n-1} \neq 0$. Then the element of $\Gamma(S)$ is of the form $u_i a^j$ where $i \leq m, ~ j \leq n$. Then $M = \{ u_i a^j : i \leq m, ~ j \leq n\}$. Since, $a^{n-1}$ is adjacent to all vertex of $M$. So, we define Roman dominating function $f = (V_0, V_1, V_2)$ such that $V_0 = M\backslash \{a^{n-1}\}, ~ V_1= \phi$ and $V_2= a^{n-1}$. Hence, every element $x$ of $V_0$ for which $f(x)= 0$ is adjacent to element of $V_2$. Thus, the Roman dominating number $\gamma_R(\Gamma(S)) = \sum_{u\in M}f(u)= \sum_{u_0\in V_0}f(u_0)+ \sum_{u_1\in V_1}f(u_1)+ \sum_{u_2\in V_2}f(u_2)= 0+0+2= 2$.
\end{proof}

\begin{theorem}
	
	Let $R= R_1 \times R_2$ be a ring such that $diam(\Gamma (R_1))= diam(\Gamma (R_2)) = 0$ and $\lvert R_1 \rvert \geq 5 ~\& ~\lvert R_2 \rvert \geq 5$. Then $\gamma_R(\Gamma(R)) = 4$.
\end{theorem}

\begin{proof}
	Let $R= R_1 \times R_2$ be a ring such that $diam(\Gamma (R_1))= diam(\Gamma (R_2)) = 0$. Then we have three cases.

	\textbf{Case 1:} $Z(R_1)= \{0, a\}$ and $Z(R_2)= \{0, b\}$ and let $Reg(R_1)= \{x_1, x_2,..., x_n\}$ and $Reg(R_2)= \{y_1, y_2,..., y_m\}$. Now, we are going to construct a graph for this case.

	\[\begin{tikzpicture}

	\vertex (A) at (-3,2) [label=below:${(x_i, b)}$]{};
	\vertex (B) at (0,2) [label=above:${(0,b)}$]{};
	\vertex (C) at (3,2) [label=right:${(x_i,0)}$]{};
	\vertex (D) at (0,0.4) [label=right:${(a,0)}$]{};
	\vertex (E) at (-3,-0.5) [label=below:${(a,b)}$]{};
	\vertex (F) at (3,-0.5) [label=below:${(0,y_i)}$]{};
	\vertex (G) at (0,-1.2) [label=below:${(a,y_i)}$]{};
	\path
	(A) edge (B)
	(B) edge (C)
	(B) edge (D)
	(B) edge (E)
	(C) edge (F)
	(D) edge (E)
	(D) edge (G)
	(D) edge (F)
	;
	
	\end{tikzpicture}\]
	\hspace{6.6cm} \textbf{Figure 2} \\
	
	Also, we define a function $g : V(\Gamma(R)) \longrightarrow \{0,1,2\}$ by
	
	\[ g(x, y) =  \left\{
	\begin{array}{ll}
	2 & if~ (x, y) = (0, b) ~and ~(x,y)= (a,0) \\
	0 & otherwise \\
	\end{array}
	\right. \]
	
	Here, it is easily seen that $g$ is a Roman dominating function such that $g(v)= 2+2=4$. Hence, $\gamma_R(\Gamma(R)) = 4$.

	\textbf{Case 2:} Suppose $R_1$ is an integral domain and $Z(R_2)= \{0,b\}$, then we have the following induced subgraph.

	\[\begin{tikzpicture}

	\vertex (A) at (-1.5,1.5) [label=left:${(x_i, b)}$]{};
	\vertex (B) at (0.5,1.5) [label=right:${(0,b)}$]{};
	\vertex (C) at (0.5,-0.5) [label=right:${(x_i,0)}$]{};
	\vertex (D) at (-1.5,-0.5) [label=left:${(0,y_i)}$]{};
	\path
	(A) edge (B)
	(B) edge (C)
	(C) edge (D)
	;
	
	\end{tikzpicture}\]
	\hspace{6.6cm} \textbf{Figure 3} \\
	
	Again, we define a function $g$ as follows:
	
	\[ g(x, y) =  \left\{
	\begin{array}{ll}
	2 & if~ (x, y) = (0, b) ~and ~(x,y)= (x_i,0)~ for~ a~ fixed~ i\\
	0 & otherwise \\
	\end{array}
	\right. \]
	
	Clearly, $g$ is a RDF with $g(v)= 2+2=4$. Therefore, $\gamma_R(\Gamma(R)) = 4$.

	\textbf{Case 3:} Now, we suppose $R_1$ and $R_2$ are integral domains. In this case, $\Gamma(R)$ is a complete bipartite graph and $\lvert R_1 \rvert \geq 5 ~\& ~\lvert R_2 \rvert \geq 5$. Therefore, $\gamma_R(\Gamma(R)) = 4$.

\end{proof}

\begin{theorem}
	
	Let $R= R_1 \times R_2$ be a ring such that $diam(\Gamma (R_1))= 0$ and $diam(\Gamma (R_2)) = 1$. Then $\gamma_R(\Gamma(R)) = 4$.
	
\end{theorem}

\begin{proof}
	Since $diam(\Gamma (R_1))= 0$ and $diam(\Gamma (R_2)) = 1$. Then we have two cases for the ring $R_1$.\\
	
	\textbf{Case 1:} Let $Z(R_1)= \{0,a\}$. Then $Reg(R_1)= \{x_1,x_2,...,x_n\}$, $Reg(R_2)= \{y_1,y_2,...,y_m\}$. Suppose $Z(R_2)= \{0,z_1,z_2,...,z_k\}$ such that $z_i z_j = 0$ for all $i,j \leq k$. Now, we are going to construct a graph for this condition.
	
	\[\begin{tikzpicture}

	\vertex (A) at (1,4) [label=above:${(a, z_j)}$]{};
	\vertex (B) at (-1,2) [label=above:${(0,z_j)}$]{};
	\vertex (C) at (3,2) [label=above:${(a,0)}$]{};
	\vertex (D) at (0,0) [label=left:${(x_i,0)}$]{};
	\vertex (E) at (2,0) [label=right:${(0,y_j)}$]{};
	\vertex (F) at (5,2) [label=below:${(a,y_i)}$]{};
	\vertex (G) at (-3,2) [label=below:${(x_i,z_j)}$]{};
	\path
	(A) edge (B)
	(A) edge (C)
	(B) edge (G)
	(B) edge (D)
	(B) edge (C)
	(C) edge (F)
	(C) edge (E)
	(D) edge (E)
	;
	
	\end{tikzpicture}\]
	\hspace{6.6cm} \textbf{Figure 4} \\
	
	Also, we define a function $g$ as follows:
	
	\[ g(x, y) =  \left\{
	\begin{array}{ll}
	2 & if~ (x, y) = (a, 0) ~and ~(x,y)= (0,z_j) for~ j=1 \\
	0 & otherwise \\
	\end{array}
	\right. \]
	
	It has been easily seen that $g$ is an RDF. Therefore, $g(v)= 2+2= 4$ and hence $\gamma_R(\Gamma(R)) = 4$.
	
	\textbf{Case 2:} Let $R_1$ be an integral domain. Then we have an induced subgraph given in fig $4$.
	
	\[\begin{tikzpicture}

	\vertex (A) at (-1,0) [label=left:${(x_i, z_j)}$]{};
	\vertex (B) at (0,0) [label=right:${(0,z_j)}$]{};
	\vertex (C) at (0,1) [label=right:${(x_i,0)}$]{};
	\vertex (D) at (0,2) [label=left:${(0,y_j)}$]{};
	\path
	(A) edge (B)
	(B) edge (C)
	(C) edge (D)
	;
	
	\end{tikzpicture}\]
	\hspace{6.6cm} \textbf{Figure 5} \\
	
	Again, we define a function $g$ as follows.
	
	\[ g(x, y) =  \left\{
	\begin{array}{ll}
	2 & if~ (x, y) = (x_i, 0) ~and ~(x,y)= (0,z_j) for~ i=j=1 \\
	0 & otherwise \\
	\end{array}
	\right. \]
	
	It can be easily verify that $g$ is an RDF. Then $g(v)= 2+2= 4$ and hence $\gamma_R(\Gamma(R)) = 4$.
	
\end{proof}

\begin{theorem}
	
	Let $R= R_1 \times R_2$ be a ring such that $diam(\Gamma (R_1))= diam(\Gamma (R_2)) = 1$. Then $\gamma_R(\Gamma(R)) = 4$.
	
\end{theorem}

\begin{proof}
	The proof is the same as the proof of the Theorem $3.3$.
	
\end{proof}

\begin{theorem}
	
	Let $R= R_1 \times R_2$ be a ring such that $diam(\Gamma (R_1))= 0$ and $diam(\Gamma (R_2)) = 2$. Then $\gamma_R(\Gamma(R)) = 4$.
	
\end{theorem}

\begin{proof}
	Let $R= R_1 \times R_2$ be a ring and $R_2$ be a finite local ring generated by $x$, say, $Z(R_2)= xR_2$ with $x^l=0$ and $x^{l-1}\neq 0$. Now, we have two cases. \\
	
	\textbf{Case 1:} Suppose $Z(R_1)= \{0,a\}$, $Reg(R_1)= \{u_1,u_2,...,u_n\}$, $Reg(R_2)= \{v_1,v_2,...,v_m\}$ and $Z(R_2)= \{0,v_1x,v_2x,...,v_mx^{l-1}\}$ such that two vertices $v_ix^j$ and $v_sx^r$ of $\Gamma (R)$ are adjacent if and only if $j+r \geq l$. Now, we define the RDF $g$ on $V(\Gamma(R))$ as follows. For any one value of $m$, $g(0, v_mx^{l-1}) = 2$ and $g(a,0)=2$ and for the remaining vertices $x,y$, let $g(x,y)=0$. It is easily seen that $g$ is an RDF and $g(v)=2+2=4$. \\
	
	\textbf{Case 2:} Let $R_1$ be an integral domain. Then $\Gamma(R)$ is an induced subgraph after deleting the vertices $(a,0), (a,v_j), (a, v_ix^j)$ for each $i~\& ~j$ from $case 1$. Now, defining RDF $g$ as $g(u_i,0)=2$ for any one of $i's,$ say, $i=1$ and $g(0, v_mx^{l-1})=2$ for $m=1$ and for the remaining vertices $(x,y)$, let $g(x,y)=0$. Then $g(v)=2+2=4$. Hence, in both cases, $\gamma_R(\Gamma(R)) = 4$.
	
\end{proof}

\begin{theorem}
	
	Let $R= R_1 \times R_2$ be a ring such that $diam(\Gamma (R_1))= 1~ or ~ 2$ and $diam(\Gamma (R_2)) = 2$. Then $\gamma_R(\Gamma(R)) = 4$.
	
\end{theorem}

\begin{proof}
	The proof is the same as given in Theorem $3.5$.
\end{proof}

\begin{remark}
	Let $R$ be a finite commutative ring with unity. If $R$ is a product of two local rings with diameters less than equal to $2$. Then Roman domination number is $4$.
	
\end{remark}

Let $G$ and $H$ be a graph. We define the Cartesian product of $G$ and $H$ to be the graph $G \Box H$ such that the vertex set of $G \Box H$ is $V(G) \times V(H)$, i.e., $\{(x,y)\vert x\in G, y\in H\}$. Also, two vertices $(x_1,y_1)$ and $(x_2,y_2)$ are adjacent in $G \Box H$ if and only if one of the following is true:
\begin{itemize}
	\item $x_1 = x_2$ and $y_1$ is adjacent to $y_2$ in $H$, or
	
	\item $y_1 = y_2$ and $x_1$ is adjacent to $x_2$ in $G$.
	
\end{itemize}

\begin{proposition}
	Let $R_1$ and $R_2$ be two rings such that $\lvert \Gamma(R_1) \rvert = m$ and $\lvert \Gamma(R_2) \rvert = n$ and having $\Delta (\Gamma(R_1))= r_1$, $\Delta (\Gamma(R_2))= r_2$. Then $\gamma_R(\Gamma(R_1) \Box \Gamma(R_2)) \leq mn-r_1-r_2+1$.
	
\end{proposition}

\begin{proof}
	Suppose $R_1$ and $R_2$ be two rings and $\Delta (\Gamma(R_1))= r_1$, $\Delta (\Gamma(R_2))= r_2$ with $\lvert \Gamma(R_1) \rvert = m$ and $\lvert \Gamma(R_2) \rvert = n$. Now, we know from the definition of Cartesian product of two graphs, $V(\Gamma(R_1) \Box \Gamma(R_2)) = mn$. Therefore, there exists a vertex $v$ in $\Gamma(R_1) \Box \Gamma(R_2)$ such that $deg(v)= r_1+r_2$. If $V_2$= \{v\}, $V_1= V- N[v]$ and $V_0= V-V_1-V_2$, then $f= (V_0, V_1, V_2)$ is a Roman dominating function with $f(V)= \lvert V_1 \rvert + 2\lvert V_2 \rvert = mn-(r_1+r_2+1)+2 = mn-r_1-r_2+1$. Hence, the weight of the function $f$ is $mn-r_1-r_2+1$ and $\gamma_R(\Gamma(R_1) \Box \Gamma(R_2)) \leq mn-r_1-r_2+1$.
	
\end{proof}

\begin{corollary}
	Suppose that total number of non-zero zero-divisor in a ring $R_1$ is $1$, say $\lvert Z(R_1)^*\rvert = 1$ and $\lvert Z(R_2) \rvert \geq 2$, then $\gamma_R(\Gamma(R_1) \Box \Gamma(R_2)) = \gamma_R(\Gamma(R_2))$, since $\Gamma(R_1) \Box \Gamma(R_2) \cong \Gamma(R_2)$.
	
\end{corollary}

Now, we give some examples.

\begin{example}
	
	Any graph $G$ has a Roman domination number equal to $2$, then a vertex of graph $G$ is adjacent to every other vertex of $G$. In paper \cite[Theorem 2.5]{ander}, it is proved that for a commutative ring $R$, there is a vertex of $\Gamma(R)$ which is adjacent to every other vertex if and only if either $R \equiv \mathbb{Z}_2 \times A$ where $A$ is an integral domain, or $Z(R)$ is an annihilator ideal (and hence is a prime).
	
\end{example}

\begin{example}
	
	\textbf{(a)}	In \cite{akbari}. it is proved that for any finite ring $R$, if $\Gamma(R)$ is a regular graph of degree $m$, then $\Gamma(R)$ is a complete graph $K_m$ or a complete bipartite graph $K_{m,m}$. In this case, $\gamma_R(\Gamma(R)) = 2~or~ 4$, provided $m\geq 3$. \\
	
	\textbf{(b)}   In \cite[Theorem 9]{akbari}, let $R$ be a finite principal ideal ring. If $\Gamma(R)$ is a Hamiltonian graph, then it is either a complete graph or complete bipartite graph. Thus $\gamma_R(\Gamma(R)) = 2~or~ 4$.\\
	
	\textbf{(c)}   In \cite[Theorem 8]{akbari} , let $R$ be a finite decomposable ring. If $\Gamma(R)$ is a Hamiltonian graph, then $\Gamma(R) \equiv K_{n,n}$ for some natural number $n$. Consequently, $\gamma_R(\Gamma(R)) = 4$.
	
\end{example}

\begin{corollary}
	
	In \cite[Corollary 1]{akbari}, the graph $\Gamma(\mathbb{Z}_n)$ is a Hamiltonian graph if and only if $n= p^2$ where $p$ is a prime greater than $3$ and in this case, $\Gamma(\mathbb{Z}_n) \equiv K_{p-1}$. Thus, Roman domination number of $\Gamma(R)$ is $2$.
	
\end{corollary}

\begin{theorem}
	
	Let $R= R_1 \times R_2 \times ...\times R_n$, for a fixed integer $n\geq 3$ and $R_i$ be an integral domain for each $i= 1,2,...,n$. Then $\gamma_R(\Gamma(R)) = 2n$.
	
\end{theorem}

\begin{proof}
	Let $R= R_1 \times R_2 \times ...\times R_n$ be a ring and each $R_i$ be an integral domain for $i= 1,2,...,n$. Then the set $S= \{(1,0,0,...,0), (0,1,0,...,0),...,(0,0,0,...,1)\}$ upto $n$ terms is a dominating set and no subset $T$ of $R_1 \times R_2 \times ...\times R_n$ with cardinality less than $n$ can be a dominating set. Now, define a $\gamma_R$- function $g$ in such a way that $g(1,0,0,...,0)= g(0,1,0,...,0)= g(0,0,1,...,0)= ...= g(0,0,0,...,1)= 2$ and $g(u)= 0$ for rest of the vertices of $\Gamma(R)$ where $u$ is a vertex of $n$-tuples of $R_1 \times R_2 \times ...\times R_n$. Therefore, $g(V)= 2+2+...+2$ upto $n$ terms, it follows that $g(V)= 2n$ and hence $\gamma_R(\Gamma(R)) = 2n$.
	
\end{proof}

It is known from \cite[Theorem 8.7]{atiyah} that a finite Artinian ring $R$ is a product of local rings. Based on this result, we have the following:

\begin{theorem}
	
	Let $R$ be a finite commutative Artinian ring and $R= R_1 \times R_2 \times ...\times R_n$ where each $R_i$ is a local ring for $i= 1,2,...,n$. Then $\gamma_R(\Gamma(R)) = 2n$ for $n\geq 3$.
	
\end{theorem}

\begin{proof}
	The proof of this Theorem is the same as above.
	
\end{proof}

\begin{theorem}
	
	Let $n= p_1^{m_1} p_2^{m_2}...p_k^{m_k}$ for any fixed integer $k\geq 1$ and for distinct primes $p_1,..., p_k$ and positive integers $m_1,..., m_k$. Then the following results hold.
	
\end{theorem}

\noindent \textbf{(a)}  $\gamma_R(\Gamma(\mathbb{Z}_n)) = 2k,$ if $n= p_1^{m_1}... p_k^{m_k}$ and $ k\geq 3$.\\

\noindent \textbf{(b)}  $\gamma_R(\Gamma(\mathbb{Z}_n)) = 4,$ if $n= p_1^{m_1} p_2^{m_2}$ where either $m_1 \geq 2~ or ~ m_2\geq 2~ \text{or}~ m_1= m_2=1$ and $p_1, p_2 \geq 3$.\\

\noindent \textbf{(c)}  $\gamma_R(\Gamma(\mathbb{Z}_n)) = 2,$ if $n= p_1^{m_1}$ and $m_1 \geq 2$ or $n= p_1 p_2$ where either $p_1 = 2, p_2 \geq 3 ~or~ p_1 \geq 3, p_2 = 2$.

\begin{proof}
	It is known that $\mathbb{Z}_n \cong \mathbb{Z}_{p_1}^{m_1} \times ... \times \mathbb{Z}_{p_k}^{m_k}$ where $n= p_1^{m_1}...p_k^{m_k}$.
	To prove part (a), from \cite[Proposition 1]{cockayne}, $\gamma(G)\leq \gamma_R(G) \leq 2\gamma(G)$. We suppose that $k \geq 3$ and $p_1, p_2,...,p_k$ are distinct primes and $m_1, m_2,..., m_k$ are positive integers. Now, we construct a set $T= \{(p_1^{m_1-1},0,0,...,0), (0, p_2^{m_2-1},0,...,0),.$ $..,(0,0,0,...,p_k^{m_k-1})\}$ and define a function $f= (V_0, V_1, V_2)$ such that $V_0 = V(\Gamma(\mathbb{Z}_n)) \setminus T$, $V_1= \phi$, $V_2 = T$. Then the function $f$ is a Roman dominating function with $f(V)= 2\lvert T\rvert = 2k$ and hence $\gamma_R(\Gamma(\mathbb{Z}_n)) = 2k$. \\
	Proof of (b) is same as part (a). To prove (c), it is noted that $p_1^{m_1-1}$ is a vertex adjacent to all vertices in $\Gamma(\mathbb{Z}_n)$ and whenever $n= p_1p_2$, where $p_1 = 2, p_2 \geq 3 ~or~ p_1 \geq 3, p_2 = 2,$ then $\Gamma(\mathbb{Z}_n)$ is a star graph. Thus $\gamma_R(\Gamma(\mathbb{Z}_n)) = 2$.
	
\end{proof}

\begin{theorem}
	
	If $R$ is a commutative ring with unity that contains at least one prime ideal, then $\gamma_R(\Gamma(R)) \leq 2(\lvert p\rvert -1)$ where $p$ is a prime ideal.
	
\end{theorem}

\begin{proof}
	Suppose $R$ is a commutative ring with unity that contains at least one prime ideal $p$, then every edge in $\Gamma(R)$ has at least one end vertex in $p$. Let $S= p\setminus \{0\}$. Then $S$ is a dominating set of $\Gamma(R)$. Since every edge has at least one vertex in $S$, we define a function $g= (V_0, V_1, V_2)$ such that $V_0= V(\Gamma(R))\setminus S$, $V_1= \phi$, $V_2= S$. Also, $p$ is a prime ideal whose cardinality is minimal among the cardinalities of all prime ideals of $R$. Hence, we can conclude $g(V) \leq 2\lvert S\rvert = 2(\lvert p\rvert -1)$. Thus, $\gamma_R(\Gamma(R)) \leq 2(\lvert p\rvert -1)$.
	
\end{proof}

\section{Roman domination number of total graph}

Let $R$ be a commutative ring with unity and $Z(R)$ be the set of zero-divisors of $R$. The total graph of a ring $R$, denoted by $T(\Gamma(R))$, is the undirected graph with vertex set $R$ and two vertices $x, y \in R$ are adjacent if and only if $x+y \in Z(R)$. This concept was introduced and studied by Anderson and Badawi \cite{anderson}. Here, $Z(\Gamma(R))$ and $Reg(\Gamma(R))$ are disjoint induced subgraphs of $T(\Gamma(R))$. It is observed that if $Z(R)$ is an ideal with $\lvert Z(R) \rvert = \beta$ and $2\in Z(R)$, then $deg(v) = \beta -1$ for every $v \in V(T(\Gamma(R)))$ and if $2 \notin Z(R)$, then $deg(v) = \beta -1$ for each $v \in Z(R)$ and $deg(v) = \beta$ for every $v \notin Z(R)$. Hence, there is no vertex in $T(\Gamma(R))$ which has degree $\lvert R \rvert - 1$. In this section, we obtain bounds for Roman domination number of $T(\Gamma(R))$.\\

To obtain the bound for the Roman domination number of $T(\Gamma(R))$, we first state the following result of Anderson and Badawi \cite{anderson}.

\begin{lemma}[\cite{anderson}, Theorem 2.1]
	Let $R$ be a commutative ring such that $Z(R)$ is an ideal of $R$. Then $Z(\Gamma(R))$ is a complete (induced) subgraph of $T(\Gamma(R))$ and $Z(\Gamma(R))$ is disjoint from $Reg(\Gamma(R))$.
	
\end{lemma}

\begin{theorem}
	Let $R$ be a commutative ring (not necessarily finite) with identity and $Z(R)$ be its ideal with $\lvert R/Z(R) \rvert = \alpha$. Then $3 \leq \gamma_R(T(\Gamma(R))) \leq 2\alpha$.
	
\end{theorem}

\begin{proof}
	It is noted that $T(\Gamma(R))$ has no vertex of degree $\lvert R \rvert -1$, therefore, $\gamma_R(T(\Gamma(R))) \geq 3$. Suppose $\lvert Z(R) \rvert = \beta$ and $Z(R)$ is an ideal of $R$. Then, by Theorem $2.2$ of \cite{anderson}, there are two cases arise: $(i)$ If $2 \in Z(R)$, then each coset $x+Z(R)$ is a complete graph $K_{\beta}$ and $(ii)$ if $2 \notin Z(R)$, then coset $(x+Z(R)) \bigcup (-x+Z(R))$ is a complete bipartite graph $K_{\beta, \beta}$. Hence,
	
	\[
	T(\Gamma(R))=
	\begin{cases}
	K_{\beta}\cup \underbracket{K_{\beta} \cup K_{\beta} \cup ... \cup K_{\beta}}_{(\alpha - 1) copies},& \text{if } 2 \in Z(R)\\
	K_{\beta}\cup \underbracket{K_{\beta, \beta} \cup K_{\beta, \beta} \cup ... \cup K_{\beta, \beta}}_{(\frac{\alpha - 1}{2}) copies},& \text{if } 2 \notin Z(R)
	\end{cases}
	\]
	
	Note that each of the component in case $(i)$ has Roman domination number $2$ as the definition of complete graph. Therefore, $\gamma_R(T(\Gamma(R)))= 2\alpha$. In case $(ii)$,  each complete bipartite graph has maximum Roman domination number $4$. Hence, $\gamma_R(T(\Gamma(R))) \leq 2+ (\frac{\alpha -1}{2}) 4 = 2\alpha$. Thus, in both cases, $3 \leq \gamma_R(T(\Gamma(R))) \leq 2\alpha$.
	
\end{proof}

\section{Conclusion}
The main objective of this paper was to study the Roman domination number of zero-divisor graph of a commutative ring $R$. Here, we have first calculated the Roman domination number of $\Gamma(R)$ where $R = R_1 \times R_2$, for different diameters of $\Gamma(R_1)$ and $\Gamma(R_2)$ and then we generalized it for $R= R_1 \times R_2 \times ... \times R_n$. At the end of this paper, we have discussed the upper and lower bounds of Roman domination number of total graph $T(\Gamma(R))$.

%\section*{Acknowledgement}
%The authors are thankful to the Department of Science and Technology (DST) (under project CRG/2020/005927, vide Diary No. SERB/F/6780/ 2020-2021 dated 31 December, 2020) for financial supports and Indian Institute of Technology Patna for providing research facilities.

%%%%%%%%%%%%%%%%%%%%%%%%%%%%%%%%

\end{document}